\documentclass[11pt,a4paper]{amsart}

\usepackage{amsmath}
\usepackage{amsthm}
\usepackage{enumerate} % enmurate as item, lists etc..
\usepackage{amssymb}
\usepackage{graphicx}
\usepackage[all]{xy}
\usepackage{hyperref}
\usepackage{aliascnt}
\usepackage{stmaryrd} %newsymbole mapsfrom, varoplus
\usepackage{mathtools}
\usepackage{verbatim} %To be able to write latex without being typesetted
\usepackage{txfonts} % new fonts
\usepackage{lmodern}
\usepackage{setspace}   %Allows double spacing with the \doublespacing command

\usepackage[hmargin=2.9cm]{geometry}
\newtheorem{lma}{Lemma}[section]

\newaliascnt{thmCt}{lma}
\newtheorem{thm}[thmCt]{Theorem}
\aliascntresetthe{thmCt}

\newaliascnt{corCt}{lma}
\newtheorem{cor}[corCt]{Corollary}
\aliascntresetthe{corCt}

\newaliascnt{propCt}{lma}
\newtheorem{prop}[propCt]{Proposition}
\aliascntresetthe{propCt}

\newtheorem*{thm*}{Theorem}
\newtheorem*{cor*}{Corollary}
\newtheorem*{prop*}{Proposition}

\theoremstyle{definition}

\newaliascnt{prgCt}{lma}
\newtheorem{prg}[prgCt]{}
\aliascntresetthe{prgCt}

\newaliascnt{dfnCt}{lma}
\newtheorem{dfn}[dfnCt]{Definition}
\aliascntresetthe{dfnCt}

\newaliascnt{rmkCt}{lma}
\newtheorem{rmk}[rmkCt]{Remark}
\aliascntresetthe{rmkCt}

\newaliascnt{rmksCt}{lma}

\aliascntresetthe{rmksCt}

\newaliascnt{ntnCt}{lma}

\aliascntresetthe{ntnCt}

\newaliascnt{ntnsCt}{lma}

\aliascntresetthe{ntnsCt}

\newaliascnt{qstCt}{lma}
\newtheorem{qst}[qstCt]{Question}
\aliascntresetthe{qstCt}

\newaliascnt{prblCt}{lma}

\aliascntresetthe{prblCt}

\newaliascnt{obsCt}{lma}

\aliascntresetthe{obsCt}

\newaliascnt{exaCt}{lma}

\aliascntresetthe{exaCt}

\newaliascnt{exasCt}{lma}

\aliascntresetthe{exasCt}

\newcommand{\T}{\mathbb{T}}

\newcommand{\N}{\mathbb{N}}

\newcommand{\R}{\mathbb{R}}
\newcommand{\Z}{\mathbb{Z}}
\newcommand{\K}{\mathrm{K}}

\newcommand{\Ca}{C^*}

\DeclareMathOperator{\Cu}{Cu}

\DeclareMathOperator{\AI}{AI}

\DeclareMathOperator{\AF}{AF}
\DeclareMathOperator{\Lsc}{Lsc}

\DeclareMathOperator{\Hom}{Hom}
\DeclareMathOperator{\id}{id}

\begin{document}
\onehalfspacing
\title{Continuous fields of interval algebras}
\author{Laurent Cantier}

\address{Laurent Cantier\newline
Departamento de Matem\'aticas\\
Universidad de Zaragoza\\
C/Pedro Cerbuna 12\\
50009 Zaragoza\\
Spain}
\email[]{lcantier@unizar.es }

\thanks{The author was supported by the Spanish State Research Agency through Consolidaci\'on Investigadora program No. CNS2022-135340 and by the European Union-NextGenerationEU through a Margarita Salas grant.}

\begin{abstract} 
This paper investigates and classifies a specific class of one-parameter continuous fields of $\Ca$-algebras, which can be seen as generalized $\AI$-algebras. Building on the classification of *-homomorphisms between interval algebras by the Cuntz semigroup, along with a selection theorem and a gluing procedure, we employ a \textquoteleft local-to-global\textquoteright\ strategy to achieve our classification result.
\end{abstract}
\maketitle
%%%%%%%%% Intro
\section{Introduction}
Continuous fields of $\Ca$-algebras have been a subject of extensive study for decades. These structures, defined by families of $\Ca$-algebras indexed over topological spaces, offer a new perspective on $\Ca$-algebras, allowing them to be analyzed through their fibers. The present paper continues this line of investigation, focusing on continuous fields of AI-algebras (inductive limits of interval algebras) and exploring their morphisms from their local behavior. We recall the reader that Elliott and Ciuperca classified *-homomorphisms from $C([0,1])$ to stable rank one $\Ca$-algebras by means of the Cuntz semigroup. See \cite{CE08}. Subsequently, Robert generalized this result to classify *-homomorphisms between unital $\AI$-algebras. See \cite{R12}. As a direct consequence, the approximate intertwining technique yield the classification of these $\Ca$-algebras via the Cuntz semigroup. Parallel to this, the Elliott-Glimm classification of $\AF$-algebras by means of their $\K_0$-groups has motivated the classification of continuous fields of $\AF$-algebras. See \cite{DEN11}.

This work addresses the following question: Can we classify continuous fields of unital $\AI$-algebras?  While we aim for broad generality, we acknowledge the presence of natural obstructions and consequent restrictions. E.g. on the dimension of the topological \textquoteleft base\textquoteright\ space or on the global behavior of the fibers. After thoroughly establish the main definitions and properties of continuous fields of $\Ca$-algebras and their morphisms, we employ a local-to-global strategy. Firstly, we extend the existence part of Robert's classification to morphisms of continuous fields using an elegant version of Michael's selection theorem. Secondly, we extend the uniqueness part using a technical gluing process. Finally, we introduce the concept of continuous fields of Cuntz semigroups, leading to a functorial classification of certain continuous fields of $\AI$-algebras.  We expose our main result below.

Let $\mathcal{C}$ denotes the class of continuous fields that are realized as $\underset{\rightarrow}{\lim}(C[0,1]\otimes A_i,\phi_{ii+1})_{i\in\N}$, where $A_i$ are unital interval algebras and $\phi_{ii+1}$ are unital continuous fields morphisms.

\begin{thm*}
Let $\mathcal{A},\mathcal{B}\in \mathcal{C}$. Let $\alpha\colon \Cu(\mathcal{A})\rightarrow \Cu(\mathcal{B})$ be any scaled $\Cu$-morphism descending to the fibers.
Then there exists a continuous field morphism $\phi\colon \mathcal{A}\rightarrow\mathcal{B}$, unique up to approximate unitary equivalence, such that $\Cu(\phi)=\alpha$. 

As a consequence, any $\Cu$-isomorphism $\alpha\colon\Cu(\mathcal{A})\simeq \Cu(\mathcal{B})$ descending to the fibers and such that $\alpha([1_\mathcal{A}])=[1_\mathcal{B}]$, lifts to a continuous field isomorphism $\mathcal{A}\simeq \mathcal{B}$ which is unique up to approximate unitary equivalence.
\end{thm*}

\textbf{Acknowledgments.} The author would like to thank S. Bhattacharjee for introducing him to the world of continuous fields of $\Ca$-algebras and M. D\u{a}d\u{a}rlat for sharing an enlightening example. Also, the referee for making pertinent comments that have improved the exposition of the manuscript.

\section{An intro on continuous fields of $\Ca$-algebras}
A continuous field of $\Ca$-algebras over a topological space $X$ is a family $\{A_x\}_{x}$ of $\Ca$-algebras indexed by $X$ varying in a continuous manner that we will make precise below. The space $X$ is often referred to as the \textquoteleft base space\textquoteright\ while the family of $\Ca$-algebras are often referred to as the \textquoteleft fibers\textquoteright. Several names have appeared in the literature (e.g., $\Ca$-bundles, $C(X)$-algebras) to shed light on different pictures of this concept. For that matter, we explicitly define what we mean by continuous fields of $\Ca$-algebras and their morphisms, based on the work \cite{N96} of Nielsen as a starting point. In the process, we establish both a $\Ca$-algebraic and a topological point of view.

Recall that by a \emph{section}, we mean a map $a\colon X\rightarrow \bigsqcup_x A_x$ such that $a(x)\in A_x$, for all $x\in X$. (Here and in the rest of this work, we freely identify $A_x$ with its image in $\bigsqcup_x A_x$.) Equivalently, a section is an element of $\prod_x A_x$.

\begin{dfn}\label{dfn:ctsfields}
A \emph{continuous field of $\Ca$-algebras} is a triple $(X,\{A_x\}_{x}, \mathcal{A})$ where $X$ is a compact Hausdorff space, $\{A_x\}_{x}$ is a family of $\Ca$-algebras indexed over $X$ and $\mathcal{A}$ is a family of sections satisfying the following conditions.

(i) $\mathcal{A}$ is a $\Ca$-algebra under pointwise and supremum norm.

(ii) For each $x\in X$, the map $\pi_x\colon \mathcal{A}\rightarrow A_x$ sending $a\mapsto a(x)$ is a surjective *-homomorphism.

(iii) For each $a\in \mathcal{A}$, the map $\nu_a\colon X\rightarrow \R$ sending $x\mapsto \Vert a(x)\Vert$ is continuous.

(iv) $\mathcal{A}$ is stable by $C(X)$-multiplication. That is, for any $\lambda\in C(X)$ and any $a\in \mathcal{A}$, the section $\lambda.a\colon x\mapsto \lambda(x) a(x)$ belongs to $\mathcal{A}$.
 \end{dfn}

Note that the above definition has a \textquoteleft$\Ca$-algebraic flavor\textquoteright\ as opposed to a more traditional picture which has a more \textquoteleft topological flavor\textquoteright. Namely, continuous fields are sometimes defined as a triple $(E,X,\pi\colon \mathcal{A}\rightarrow X)$ where $E$ and $X$ are topological spaces (respectively, the global and the base space) and $\pi$ is a continuous surjection satisfying specific properties. In this picture, the fibers would consist of the family $\{\pi^{-1}(\{x\})\}_x$. See \cite{AM10} for more on this. 
 
As it happens, the conditions in our definition ensures that the field has enough sections -by (ii)-, which are well-behaved -by (iii)-, to define a relevant topology on the global space $\bigsqcup_x A_x$. 
More specifically, we consider $\tau_{\mathcal{A}}$ to be the topology on $\bigsqcup_x A_x$ generated by the basis of open sets of the form
\[
V_{\epsilon}^c:=\{(\xi,x) \in \bigsqcup_{x\in X} A_x \mid x\in V \text{ and } \Vert \xi-c(x)\Vert< \epsilon\}
\] 
where $V\subseteq X$ is an open set, $c\in \mathcal{A}$ and $\epsilon >0$. (Again, we have freely identified $A_x$ with its image in $\bigsqcup_x A_x$, and the above norm is computed in $A_x$.)

This topology allows us to identify elements of $\mathcal{A}$ with $\tau_{\mathcal{A}}$-continuous sections of the canonical surjection $\bigsqcup_x A_x\relbar\joinrel\twoheadrightarrow X$. More specifically, we obtain the following continuity criterion.

\begin{prop}[Continuity Criterion]\label{prop:cc}
Let $(X,\{A_x\}_{x}, \mathcal{A})$ be a continuous field of $\Ca$-algebras and let $a\colon X\rightarrow \bigsqcup_x A_x$ be a section. Then the following are equivalent.

(i) $a \in \mathcal{A}$.

(ii) For any $x\in X$ and any $\epsilon>0$, there exist an open neighborhood $V$ of $x$ and some $c\in\mathcal{A} $ such that $a(v)\in V_\epsilon^c$, for all $v\in V$.

(iii) $a$ is continuous with respect to $\tau_{\mathcal{A}}$. \end{prop}

\begin{proof} The implication (i) $\Rightarrow$ (ii) is obvious. 

Assume that $a$ satisfies (ii). Consider a basic open set $V_\epsilon^c$ such that $a^{-1}(V_\epsilon^c)\neq \emptyset$. Let $v\in a^{-1}(V_\epsilon^c)$ and write $C=\|a(v)-c(v)\|<\epsilon$. Fix $\delta>0$ such that $3\delta+C<\epsilon$. Using the assumption, we know that there exists $d\in \mathcal{A}$ and an open neighborhood $W$ of $v$ such that $\|a(w)-d(w)\|<\delta$, for any $w\in W$. We deduce that $\|(c-d)(v)\|\leq C +\delta$. By the continuity of $\nu_{(c-d)}$, we can find an open neighborhood $Z$ of $v$ such that $\|(c-d)(z)\|\leq C +2\delta$ for all $z\in Z$. Finally, we compute that $\|a(y)-c(y)\|\leq \delta+C+2\delta<\epsilon$, for all $y \in Z\cap W$. As a result, we see that $Z\cap W\subseteq a^{-1}(V_\epsilon^c)$ and we conclude that $a$ satisfies (iii).

Assume that $a$ satisfies (iii). Let $x\in X$ and $\epsilon>0$. By the surjectivity of $\pi_x\colon \mathcal{A}\rightarrow A_x$, we can find some $c\in \mathcal{A}$ such that $c(x)=a(x)$. By the continuity of $a$, we deduce that $V:=a^{-1}(X_\epsilon^c)$ is an open neighborhood of $x$ such that $a(v)\in V^c_\epsilon$ for all $v\in V$, yielding (ii).

Assume that $a$ satisfies (ii), again. Let $\epsilon>0$. By compactness of $X$, there exist a finite open cover $\{V_k\}_1^l$ of $X$ together with finitely many elements $\{c_k\}_1^l$ of $\mathcal{A}$ such that for any $v\in X$, we have $a(v)\in (V_k)_\epsilon^{c_k}$ for some $k$. Consider a partition $\{\lambda_k\}_1^l$ subordinated to the open cover $\{V_k\}_1^l$ and construct $c_\epsilon:=\sum_1^l \lambda_kc_k$. We have that $c_\epsilon\in \mathcal{A}$ and that $\|c_\epsilon(x)-a(x)\|<\epsilon$ for all $x\in X$. A standard completeness argument yields (i).
\end{proof}

We refer to (ii) as the \emph{continuity criterion}. We highlight which properties of \autoref{dfn:ctsfields} have been used in the different implications between the above equivalent statements.
\[
\xymatrix{
\text{Membership of }\mathcal{A}\,\ar@<+0,7ex>[rr]&&\,\text{Continuity criterion}\quad\ar@<+0,7ex>[rr]^{\{\nu_a\} \text{ are continuous}}\ar@<+0,7ex>[ll]^{\text{Convexity of }\mathcal{A}}&& \quad\text{Continuity w.r.t. } \tau_{\mathcal{A}}\ar@<+0,7ex>[ll]^{\{\pi_x\} \text{ are surjective}}
}
\]

\begin{rmk}
The topology $\tau_{\mathcal{A}}$ induces the norm-topology on each $A_x$. See \cite{AM10}.
\end{rmk}

We now define morphisms between continuous fields of $\Ca$-algebras, again following the \textquoteleft$\Ca$-algebraic flavor\textquoteright.

\begin{dfn}
\label{dfn:dsc}
Let $(X,\{A_x\}_{x}, \mathcal{A})$ and $(X,\{B_x\}_{x}, \mathcal{B})$ be continuous fields of $\Ca$-algebras over the same space $X$. A \emph{continuous field morphism} is a *-homomorphism $\phi\colon \mathcal{A}\rightarrow \mathcal{B}$ which induces a family of *-homomorphisms $\{\phi_x\colon A_x\rightarrow B_x\}_x$ such that the following diagram commutes for all $x\in X$
\[
\xymatrix{
\mathcal{A}\ar[d]_{\pi_x^\mathcal{A}}\ar[r]^{\phi}&\mathcal{B}\ar[d]^{\pi_x^\mathcal{B}}\\
A_x\ar[r]_{\phi_x}&B_x
}
\]
We say that $\phi$ \emph{descends to the fibers}.
\end{dfn}

The following proposition shows a couple of statements that are respectively sufficient and necessary for a field of *-homomorphisms to be a continuous field morphism, which both turn out to be equivalent. Finally, we mention that these morphisms induce continuous maps between the global spaces of the continuous fields at hand.

\begin{prop}\label{prop:eqmorphism}
Let $(X,\{A_x\}_{x},\mathcal{A})$ and $(X,\{B_x\}_{x}, \mathcal{B})$ be continuous fields of $\Ca$-algebras. Let $(\phi_x\colon A_x\rightarrow B_x)_x$ be a family of *-homomorphisms. The following are equivalent.

(i) The map $\phi\colon \mathcal{A}\rightarrow \mathcal{B}$ sending $a\mapsto (x\mapsto \phi_x(a(x)))$ is a well-defined *-homomorphism. 

(ii) The map $\phi\colon \mathcal{A}\rightarrow \mathcal{B}$ sending $a\mapsto (x\mapsto \phi_x(a(x)))$ is a well-defined *-homomorphism descending to the fibers and preserving $C(X)$-multiplication.

(iii) For any $a\in \mathcal{A}$, the section $(\phi_x(a(x)))_x$ belongs to $\mathcal{B}$.\\
Furthermore, any of the above implies 

(iv) The map $\tilde{\phi}\colon \bigsqcup_x A_x\rightarrow \bigsqcup_x B_x$ sending $(\xi,x)\mapsto (\phi_x(\xi),x)$ is $(\tau_{\mathcal{A}},\tau_{\mathcal{B}})$-continuous.
\end{prop}

\begin{proof} It is immediate that (ii) $\Rightarrow$ (i) $\Rightarrow$ (iii). Assume (iii). Then the map $\phi\colon \mathcal{A}\rightarrow \mathcal{B}$ is well-defined and $\phi$ is defined component-wise. Therefore it is immediate that $\phi$ is a *-homomorphism preserving the $C(X)$-multiplication.

Let us check that (ii) implies (iv). Let $V_\epsilon^b$ be a basic open set of $\tau_{\mathcal{B}}$ and let $(\xi,v)\in \tilde{\phi}^{-1}(V_\epsilon^b)$. We know that $\|\phi_v(\xi)-b(v)\|<\epsilon$. By the surjectivity of $\pi_v\colon \mathcal{A}\rightarrow A_v$, we can find some $a\in \mathcal{A}$ such that $a(v)=\xi$. By assumption $c:=(\phi_x(a(x)))_x$ is an element of $\mathcal{B}$. By the continuity of $\nu_{c-b}$, we conclude that there exists an open neighborhood $W$ of $v$ such that $W_\delta^a\subseteq \tilde{\phi}^{-1}(V_\epsilon^b)$ for some $\delta>0$ small enough, from which the result follows.
\end{proof}

We end this section by constructing the category $\mathcal{F}_{\Ca}$ to be the category of continuous fields of $\Ca$-algebras with continuous field morphisms. To ease notations, we may denote the set of morphisms from $(X,\{A_x\}_{x}, \mathcal{A})$ to $(X,\{B_x\}_{x}, \mathcal{B})$ by $\Hom_{\mathcal{F}_{\Ca}}(\mathcal{A},\mathcal{B})$ whenever the context is clear.

We remark that we only have defined morphisms between continuous fields over the same base space. Therefore, we fix the set of morphisms between continuous fields over different base spaces to be empty. It should be stated that more general constructions have been considered in the literature but we restrict ourselves to this setting that suits our needs.

\section{From local to global}
Based on the works of Ciuperca, Elliott, Robert and Santiago, it is known the Cuntz semigroup classifies *-homo\-morphisms from a class of unital $\Ca$-algebras with trivial $ \K_1$-groups to stable rank one $\Ca$-algebras. See \cite{CE08,RS10,R12}. Let us recall a weaker version that will suffice our needs. 

\begin{thm}\label{thm:RR}{\rm(\cite[Theorem 1.0.1]{R12})} Let $A,B$ be unital separable $\AI$-algebras and let $\alpha\colon\Cu(A)\rightarrow \Cu(B)$ be a $\Cu$-morphism satisfying $\alpha([1_A])\leq [1_B]$. 

Then there exists a *-homomorphism $\phi\colon A\rightarrow B$, unique up to approximate unitary equivalence, such that $\Cu(\phi)=\alpha$.
\end{thm}

In the aim of extending the above result to continuous fields morphisms, let us introduce a notion of $\Cu$-morphisms \emph{descending to the fibers} analogous as the one for *-homomorphisms presented in \autoref{dfn:dsc}. 

From now on, $X$ is a compact Hausdorff metric space of dimension at most one.

\begin{dfn}
\label{dfn:dscCu}
Let $(X,\{A_x\}_{x}, \mathcal{A})$ and $(X,\{B_x\}_{x}, \mathcal{B})$ be continuous fields of $\Ca$-algebras over the same space $X$. 

A $\Cu$-morphism $\alpha\colon \Cu(\mathcal{A})\rightarrow \Cu(\mathcal{B})$ \emph{descends to the fibers}, if it induces a family \\$\{\alpha_x\colon \Cu(A_x)\rightarrow \Cu(B_x)\}_x$ such that the following diagram commutes for all $x\in X$
\[
\xymatrix{
\Cu(\mathcal{A})\ar[d]_{\Cu(\pi_x^\mathcal{A})}\ar[r]^{\alpha}&\Cu(\mathcal{B})\ar[d]^{\Cu(\pi_x^\mathcal{B})}\\
\Cu(A_x)\ar[r]_{\alpha_x}&\Cu(B_x)
}
\]
\end{dfn}

We are aiming to lift any $\Cu$-morphism $\alpha\colon \Cu(\mathcal{A})\rightarrow \Cu(\mathcal{B})$ descending to the fibers, to a unique (up to an equivalence relation) continuous field morphism $\phi\colon \mathcal{A}\rightarrow \mathcal{B}$. 

\subsection{The existence} We aim to extend the existence part of \autoref{thm:RR} to morphisms between continuous fields whose fibers fall into the class of separable unital $\AI$-algebras. We make use of a very elegant result obtained by Micheal in \cite{M56}, often referred to as (one of) the Michael's selection theorem that we recall below. 

Recall that a \emph{multi-valued function} is an assignment $\Theta\colon X\rightarrow 2^Y $ where $X,Y$ are topological spaces and $2^Y$ denotes the set of all (non-empty) subsets of $Y$. A \emph{continuous selection for $\Theta$} is a continuous map $\phi\colon X\rightarrow Y$ such that $f(x)\in \Theta(x)$ for any $x\in X$. Finally, the multi-valued function $\Theta$ is said to be \emph{lower-semicontinuous}, if for any open set $V\subseteq Y$, the set $\{x\in X\mid \Theta(x)\cap V\neq \emptyset\}$ is open in $X$. 

Below is a weaker version of the selection theorem that will suffice our needs. Let us mention that a similar version of the theorem has been used in \cite{FV25}.

\begin{thm}\cite[Theorem 1.2]{M56}\label{thm:selection} Let $X$ be a compact Hausdorff space of dimension (at most) one. Let $(Y,d)$ be a metric space and let $\mathcal{S}$ be a family of closed non-empty subsets of $Y$ satisfying the following.

(i) For all $S\in \mathcal{S}$, we have that $(S,d)$ is (path-connected and) complete.

(ii) For any $\epsilon>0$, there exists $\delta>0$ such that for any $S\in \mathcal{S}$ and any $x,y\in S$ with $d(x,y)<\delta$, then there exists a continuous path in $S$ of diameter at most $\epsilon$ joining $x$ to $y$.

Then, any lower-semicontinuous multi-valued map $\Theta\colon X\rightarrow \mathcal{S}$ admits a continuous selection $\phi\colon X\rightarrow Y$.
\end{thm}

We may refer to (ii) as the \emph{uniform path property of $\mathcal{S}$}. Observe that the uniform path property already implies the path-connectedness of each $S\in \mathcal{S}$.

Let us outline how we will invoke Michael's selection theorem in our context, yielding the existence of a continuous field morphism $\phi\colon \mathcal{A}\rightarrow \mathcal{B}$ lifting a given $\Cu$-morphism $\alpha\colon \Cu(\mathcal{A})\rightarrow \Cu(\mathcal{B})$ descending to the fibers.

Firstly, we define a topology $\tau_{(\mathcal{A}\rightarrow \mathcal{B})}$ on the global space $\bigsqcup_x \Hom_{\Ca}(A_x,B_x)$, which induces the point-norm topology on each \textquoteleft fiber\textquoteright\ $\Hom_{\Ca}(A_x,B_x)$. We then show that this topology is metrizable under suitable separability conditions of the \textquoteleft fibers\textquoteright\ $A_x$. (The metric space obtained will play the role of $Y$ in the selection theorem.) Subsequently, under a suitable \emph{projective property}, we exhibit a continuity criterion for elements in $\Hom_{\mathcal{F}_{\Ca}}(\mathcal{A},\mathcal{B})$, based on the topology $\tau_{(\mathcal{A}\rightarrow \mathcal{B})}$, similarly as done in \autoref{prop:cc} for elements of $\mathcal{A}$. 

Secondly, we consider $\mathcal{S}_\alpha(x):=\{\phi_x \in \Hom_{\Ca}(A_x,B_x) \mid \Cu(\phi_x)=\alpha_x\}$ for any $x\in X$. By Robert's result, we deduce that these sets are non-empty and complete. Path-connectedness will follow from a result of \cite{A99}. 

Lastly, under assumptions on the continuous fields at hand, we show the uniform path property of $\mathcal{S}_\alpha$ and the lower-semicontinuity of the map sending $x\mapsto \mathcal{S}_\alpha(x)$. We finally apply the selection theorem to conclude our existence result via the continuity criterion. 

We end the section with a discussion on specific cases where the uniform path property and the projective property are automatically satisfied.\\

$\bullet$ \textbf{The Hom-bundle and its metrizable topology.}  Recall that a net $(\phi_\lambda\colon A\rightarrow B)_{\lambda\in\Lambda}$ of *-homomorphisms between $\Ca$-algebras $A,B$ \emph{converges in point-norm} to a *-homomorphism $\phi\colon A\rightarrow B$ if 
\[
\|\phi_{\lambda}(a)-\phi(a)\|_{B}\underset{\lambda}{\longrightarrow} 0 \text{ for any } a\in A.
\]
Equivalently, we can take $\Lambda$ to be the net of pairs $(F,\epsilon)$ where $F\subseteq A$ is a finite set and $\epsilon>0$ with the relation that $(F,\epsilon)\lesssim (F',\epsilon')$ whenever $F\subseteq F'$ and $\epsilon'\leq \epsilon$. In this setting, we say that $\phi,\psi$ \emph{agree on $F$ up to $\epsilon$}, and we write $\phi\simeq_{(F,\epsilon)}\psi$, whenever $\|\phi(a)-\psi(a)\|_{B}<\epsilon$ for any $ a\in F$. 
As a consequence, the sets
\[
(\phi,F,\epsilon):=\{\psi \in\Hom_{\Ca}(A,B) \mid \phi\simeq_{(F,\epsilon)}\psi\}
\]
form a basis of open sets for the point-norm topology on $\Hom_{\Ca}(A,B)$. 

If moreover $A$ is a separable $\Ca$-algebra, the space $(\Hom_{\Ca}(A,B),\tau_{pn})$ (where $\tau_{pn}$ denotes the point-norm topology as defined above) becomes completely metrizable. We give below an explicit construction of such a complete metric, for the reader's convenience.  

Let $(a_k)_{k\in \N}$ be a dense subset of self-adjoint elements of $A$. For any $\phi,\psi\colon A\rightarrow B$, we set 
\[
d(\phi,\psi)=\sum_{k=0}^\infty \frac{\|\phi(a_k)-\psi(a_k)\|}{2^k \|a_k\|}
\]

\begin{prop}
The above metric $d$ induces the point-norm topology $\tau_{pn}$ and turns $(\Hom_{\Ca}(A,B),d)$ into a complete metric space. 
\end{prop}

\begin{proof}
The fact that the metric $d$ is well-defined and induces $\tau_{pn}$ can be checked via standard arguments and is left to the reader. Regarding completeness, we observe that any $d$-Cauchy sequence $(\phi_n)_n$ induces a $\|.\|_B$-Cauchy sequence $(\phi_n(a_k))_n$ for any self-adjoint $a_k\in A$ involved in the construction of $d$. As a consequence, we define $\phi\colon A\rightarrow B$ via the assignment $a_k\mapsto \lim_n(\phi_n(a_k))$ that we have extended continuously from the dense sequence $(a_k)_k$ to $A$. It is readily seen that $\phi$ is a linear map preserving the *-operation. Finally, for any $\epsilon>0$, any $k,l\in \N$, we can find $m_\epsilon$ such $\|\phi(a_k)\phi(a_l)-\phi_{m_\epsilon}(a_k)\phi_{m_\epsilon}(a_l) \|<\epsilon/2$ and $\|\phi_{m_\epsilon}(a_ka_l) -  \phi(a_ka_l) \| <\epsilon/2$ which yields that $\phi$ is multiplicative on the dense subset $\{a_k\}_k$. We conclude that $\phi$ is a *-homomorphism by a standard continuity argument.  
\end{proof}

We are now defining the \emph{Hom-bundle topology}. Let $(X,\{A_x\}_{x}, \mathcal{A})$ and $(X,\{B_x\}_{x}, \mathcal{B})$ be continuous fields of $\Ca$-algebras over a compact Hausdorff metric space. Similarly as done for the global space of a continuous field, we wish to define a relevant topology on the global space $\bigsqcup_x \Hom_{\Ca}(A_x,B_x)$ and obtain a continuity criterion. 

We consider $\tau_{(\mathcal{A}\rightarrow \mathcal{B})}$ to be the topology on $\bigsqcup_x \Hom_{\Ca}(A_x,B_x)$ generated by the basis of open sets of the form
\[
V_{(F,\epsilon)}^\chi:=\{ (\xi,x)\in \bigsqcup_{x\in X} \Hom_{\Ca}(A_x,B_x)\mid x\in V \text{ and } \xi\simeq_{(F,\epsilon)}\chi_x\}
\] 
where $V\subseteq X$ is an open set, $\chi\in \Hom_{\mathcal{F}_{\Ca}}(\mathcal{A},\mathcal{B})$, $F\subset \mathcal{A}$ is a finite set and $\epsilon >0$.

This topology will allow us to link elements of $\Hom_{\mathcal{F}_{\Ca}}(\mathcal{A},\mathcal{B})$ with $\tau_{(\mathcal{A}\rightarrow \mathcal{B})}$-continuous sections of the canonical surjection $\bigsqcup_x \Hom_{\Ca}(A_x,B_x)\relbar\joinrel\twoheadrightarrow X$. In the process, we are obtaining a continuity criterion alike to the one of \autoref{prop:cc}. First, we need the following.

\begin{dfn}
Let $(X,\{A_x\}_{x}, \mathcal{A})$ and $(X,\{B_x\}_{x}, \mathcal{B})$ be continuous fields of $\Ca$-algebras. We say that $\Hom_{\mathcal{F}_{\Ca}}(\mathcal{A},\mathcal{B})$ has the \emph{projective property} if 
\[
\mathrm{(PP)}: \quad\forall\, x\in X,\forall\,  \phi_x\in \Hom_{\Ca}(A_x,B_x), \exists \chi \in \Hom_{\mathcal{F}_{\Ca}}(\mathcal{A},\mathcal{B}) \text{ such that } \chi_x=\phi_x.
\]
\end{dfn}

\begin{lma}\label{lma:approxextendmorph}
Let $\phi,\psi\in \Hom_{\mathcal{F}_{\Ca}}(\mathcal{A},\mathcal{B})$ be continuous field morphisms between continuous fields $(X,\{A_x\}_{x}, \mathcal{A})$ and $(X,\{B_x\}_{x}, \mathcal{B})$. Let $F\subseteq \mathcal{A}$ be a finite set and let $\epsilon>0$. 

Assume that $\phi_x\simeq_{(F_x,\epsilon)}\psi_x$, for some $x\in X$.
Then there exists an open neighborhood $V$ of $x$ such that $\phi_v\simeq_{(F_v,\epsilon)}\psi_v$, for all $v\in V$.
\end{lma}

\begin{proof}
Let $C:=\max_{a\in F}\|\phi(a)(x)-\psi(a)(x)\|$. Observe that $C<\epsilon$. Now, fix $\delta>0$ such that $\delta+C<\epsilon$. By continuity of the norm and finiteness of $F$, we can find an open neighborhood $V$ of $x$ such that $\|\phi(a)(v)-\psi(a)(v)\|<C+\delta<\epsilon$ for any $v\in V$ and any $a\in F$. 
\end{proof}

We are now ready for the continuity criterion for continuous field morphisms. (Again, we freely identify $\Hom_{\Ca}(A_x,B_x)$ with its image in $\bigsqcup_x \Hom_{\Ca}(A_x,B_x)$.)

\begin{prop}[Continuity Criterion]\label{prop:cc2}
Let $(X,\{A_x\}_{x}, \mathcal{A})$ and $(X,\{B_x\}_{x}, \mathcal{B})$ be continuous fields of $\Ca$-algebras. Let $(\phi_x)_{x}$ be a family of *-homomorphisms. Consider the following statements.

(i) The map $\phi\colon \mathcal{A}\rightarrow  \mathcal{B}$ sending $a\rightarrow (x\mapsto \phi_x(a(x)))$ is a well-defined continuous field morphism.

(ii) For any $x\in X$, any $F\subseteq \mathcal{A}$ finite and any $\epsilon>0$, there exist an open neighborhood $V$ of $x$ and some $\chi\in\Hom_{\mathcal{F}_{\Ca}}(\mathcal{A},\mathcal{B})$ such that $\phi_v\in V_{(F,\epsilon)}^\chi$, for all $v\in V$.

(iii) The map from $X\rightarrow \bigsqcup_x \Hom_{\Ca}(A_x,B_x)$ sending $x\mapsto \phi_x$ is continuous with respect to $\tau_{(\mathcal{A}\rightarrow \mathcal{B})}$. \\

Then (i) and (ii) are equivalent and imply (iii). The converse holds whenever the set $\Hom_{\mathcal{F}_{\Ca}}(\mathcal{A},\mathcal{B})$ satisfies $\rm (PP)$.

\end{prop}

\begin{proof}
 Let us write $f\colon x\mapsto \phi_x$ during the proof. The implication (i) $\Rightarrow$ (ii) is obvious. 

Assume that $(\phi_x)_x$ satisfies (ii). Let $a\in\mathcal{A}$. By assumption, we know that for any $x\in X$ and any $\epsilon>0$, there exists $\chi\in\Hom_{\mathcal{F}_{\Ca}}(\mathcal{A},\mathcal{B})$ and an open neighborhood $V$ of $x$ such that $\|\chi(a)(v)-\phi_v(a(v))\|<\epsilon$, for all $v\in V$. Consequently, for any $x\in X$ and any $\epsilon>0$, there exist $b:=\chi(a)\in \mathcal{B}$ and an open neighborhood $V$ of $x$ such that $\phi_v(a(v))\in V_\epsilon^b$ for all $v\in V$. By the continuity criterion of \autoref{prop:cc}, we deduce that $(\phi_x(a(x)))_x\in \mathcal{B}$ for any $a\in \mathcal{A}$. The equivalence between (i) and (ii) now follows from \autoref{prop:eqmorphism}.

Assume that $(\phi_x)_x$ satisfies (ii), again. Consider a basic open set $V_{(F,\epsilon)}^\chi$ such that $f^{-1}(V_{(F,\epsilon)}^\chi)\neq \emptyset$. Let $v\in f^{-1}(V_{(F,\epsilon)}^\chi)$. We have $\phi_v\simeq_{(F_v,\epsilon)}\chi_v$. Write $C:=\max_{a\in F}\|\phi_v(a(v))-\chi(a)(v)\|$. Observe that $C<\epsilon$. Fix $\delta>0$ be such that $3\delta+C<\epsilon$. Using the assumption, we know that there exist $\psi\in \Hom_{\mathcal{F}_{\Ca}}(\mathcal{A},\mathcal{B})$ and an open neighborhood $W$ of $v$ such that $\phi_w\simeq_{(F_w,\delta)}\psi_w$ for all $w\in W$. We deduce that $\chi_v\simeq_{(F_v,C+\delta)}\psi_v$. By \autoref{lma:approxextendmorph}, we can find an open neighborhood $Z$ of $v$  such that $\chi_z\simeq_{(F_z,C+2\delta)}\psi_z$ for all $z\in Z$. Finally, we compute that $\phi_y\simeq_{(F_y,C+3\delta<\epsilon)}\chi_y$ for any $y\in Z\cap W$. As a result, we see that $Z\cap W\subseteq f^{-1}(V_{(F,\epsilon)}^\chi)$ and we conclude that $(\phi_x)_x$ satisfies (iii).

Assume that $(\phi_x)_x$ satisfies (iii) and that $\Hom_{\mathcal{F}_{\Ca}}(\mathcal{A},\mathcal{B})$ satisfies $\rm (PP)$. Let $x\in X$, let $F\subseteq \mathcal{A}$ be a finite set and let $\epsilon>0$. Using $\rm (PP)$, we can find some $\chi\in \Hom_{\mathcal{F}_{\Ca}}(\mathcal{A},\mathcal{B})$ such that $\chi_x=\phi_x$. By the continuity of $f$, we deduce that $V:=f^{-1}(X_{(F,\epsilon)}^\chi)$ is an open neighborhood of $x$ such that $f(v)\in V_{(F,\epsilon)}^\chi$ for all $v\in V$, yielding (ii).
\end{proof}

We highlight the above implications in the following diagram.
\[
\xymatrix{
\text{Membership of }\Hom_{\mathcal{F}_{\Ca}}(\mathcal{A},\mathcal{B})\ar@<+0,7ex>[r]&\text{Continuity criterion}\ar@<-0,7ex>[d]^{}\ar@<0,7ex>[l]^{}\\
&\text{Continuity w.r.t. } \tau_{(\mathcal{A}\rightarrow \mathcal{B})}\ar@<-0,7ex>[u]_{\rm (PP) }
}
\]

\begin{rmk}
(i) The topology $\tau_{(\mathcal{A}\rightarrow \mathcal{B})}$ induces the point-norm topology on each Hom-set $\Hom_{\Ca}(A_x,B_x)$.

(ii) In the case that $(X,\{A_x\}_{x}, \mathcal{A})$ is a continuous field of separable $\Ca$-algebras, the topology $\tau_{(\mathcal{A}\rightarrow \mathcal{B})}$ is generated by the basis of open sets of the form 
\[
V_{\epsilon}^\chi:=\{ (\xi,x)\in \bigsqcup_{x\in X} \Hom_{\Ca}(A_x,B_x)\mid x\in V \text{ and } d_x(\xi,\chi_x)<\epsilon\}
\] 
where $V\subseteq X$ is an open set,  $\epsilon >0$, $\chi\in \Hom_{\mathcal{F}_{\Ca}}(\mathcal{A},\mathcal{B})$ and $d_x$ is any complete metric on $\Hom_{\Ca}(A_x,B_x)$ as constructed before. 
\end{rmk}

\begin{prg}\label{prg:maitrise}Let us now turn to metrizability of $\tau_{(\mathcal{A}\rightarrow \mathcal{B})}$. Assume that $(X,\{A_x\}_{x}, \mathcal{A})$ is a continuous field of separable $\Ca$-algebras over a compact Hausdorff metric space $(X,d_X)$. The Hom-bundle $(\bigsqcup_x \Hom_{\Ca}(A_x,B_x),\tau_{(\mathcal{A}\rightarrow \mathcal{B})})$ becomes metrizable as follows. 

Fix any complete metric $d_x$ on the set $\Hom_{\Ca}(A_x,B_x)$ for any $x\in X$. For any $(\xi,v),(\zeta, w)\in \bigsqcup_x \Hom_{\Ca}(A_x,B_x)$, we set 
\[
d((\xi,v),(\zeta, w))=\inf_{\chi\in \Hom_{\mathcal{F}_{\Ca}}(\mathcal{A},\mathcal{B})}\{d_\chi((\xi,v),(\zeta, w))\}+d_X(v,w).
\]
where $d_\chi((\xi,v),(\zeta, w)):= \max(d_v(\xi,\chi_v),d_w(\zeta,\chi_w))$.
It can be checked that $d$ is a well-defined metric on $\Hom_{\Ca}(A,B)$. 
\end{prg}

\begin{prop}
The Hom-bundle $(\bigsqcup_x \Hom_{\Ca}(A_x,B_x),\tau_{(\mathcal{A}\rightarrow \mathcal{B})})$ is metrizable whenever $X$ is a second-countable compact Hausdorff space and $(X,\{A_x\}_{x}, \mathcal{A})$ has separable fibers.

More precisely, any of the above metrics $d$ induces $\tau_{(\mathcal{A}\rightarrow \mathcal{B})}$. 
\end{prop}

\begin{proof}
The fact that the metric $d$ is well-defined and induces $\tau_{(\mathcal{A}\rightarrow \mathcal{B})}$ can be checked via standard arguments and left to the reader.
\end{proof}

$\bullet$ \textbf{Application of the selection theorem.} For the subsequent paragraph, we fix the following that we will refer to as \emph{the current setting}.

We let $(X,\{A_x\}_{x}, \mathcal{A})$ and $(X,\{B_x\}_{x}, \mathcal{B})$ be continuous fields of separable $\Ca$-algebras over the same compact Hausdorff metric space $(X,d_X)$ of dimension \emph{at most one}.
We fix metrics $\{d_x\}_x$ and $d$ on $\{\Hom_{\Ca}(A_x,B_x)\}_x$ and $\bigsqcup_x \Hom_{\Ca}(A_x,B_x)$ respectively as in \autoref{prg:maitrise}. (Recall that $d$ induces $\tau_{(\mathcal{A}\rightarrow \mathcal{B})}$ and $d_x$ induces $\tau_{pn}$.)

We let $\alpha\colon \Cu(\mathcal{A})\rightarrow \Cu(\mathcal{B})$ be any $\Cu$-morphism which descends to the fibers and such that $\alpha_x([1_{A_x}])=[1_{B_x}]$. We consider the multi-valued map
\[
\begin{array}{ll}
\Theta_\alpha\colon X\longrightarrow \mathcal{S}_\alpha\\
\hspace{0,9cm}x\longmapsto \{\phi_x\in \Hom_{\Ca}(A_x,B_x)\mid \Cu(\phi_x)=\alpha_x\}
\end{array}
\] 
where $\mathcal{S}_\alpha:=\{\Theta_\alpha(x)\}_{x\in X}$ is seen as a family of subsets of $(\bigsqcup_x \Hom_{\Ca}(A_x,B_x),d)$.

\begin{prop}\label{prop:(i)} Retain the current setting. Then the family $\mathcal{S}_\alpha$ satisfies condition (i) of Michael's selection theorem. (See \autoref{thm:selection}.)
\end{prop}

\begin{proof}
We deduce from \autoref{thm:RR} that $\Theta_{\alpha}(x)$ is a closed non-empty subset of\break $(\Hom_{\Ca}(A_x,B_x),d_x)$. Since the restriction of $d$ on $\Hom_{\Ca}(A_x,B_x)$ is $d_x$, it means that $\Theta_{\alpha}(x)$ is a non-empty closed subset $ (\bigsqcup_x \Hom_{\Ca}(A_x,B_x),d)$, and hence $(\Theta_{\alpha}(x),d_x)$ is a complete space for any $x\in X$. 
We also know from \cite[Theorem 2.4]{A99}, that $(\Theta_{\alpha}(x),d_x)$ is path-connected for any $x\in X$, which ends the proof.
\end{proof}

We now turn to the uniform path property of $\mathcal{S}_\alpha$ and the lower-semicontinuity of $\Theta_\alpha$. Before diving into the proof, we will need the following results.

\begin{prg}[Admissible sets]\label{prg:uniform} Recall that $\Cu$-morphisms $\alpha,\beta\colon S\rightarrow T$ \emph{compare} on a (finite) set $G\in S$, and we write $\alpha\simeq_{G}\beta$, if $\alpha(g')\leq \beta(g)$ and $\beta(g')\leq \alpha(g)$ for any $g',g\in G$ with $g'\ll g$. (See \cite[Definition 3.9]{C22}.)

 It has been shown in \cite[Theorem 3.3.1]{R12}, that the classification of *-homomorphisms between unital $\AI$-algebras (recalled in \autoref{thm:RR}) is \emph{uniformly continuous}. More specifically, given a unital $\AI$-algebra $A$, a finite set $F\subseteq A$ and some arbitrary precision $\epsilon>0$, there exists a finite set $G\subseteq \Cu(A)$ satisfying the following: 

For any (unital) *-homomorphisms $\phi,\psi\colon A\rightarrow B$ into a unital $\AI$-algebra $B$, there exists  $u\in U(B)$ such that $\phi\simeq_{(F,\epsilon)} u\psi u^*$, whenever $\Cu(\phi)\simeq_G\Cu(\psi)$.

We refer to $G$ as a \emph{$(F,\epsilon)$-admissible set for $A$}.
\end{prg}

\begin{lma}\label{lma:approxCumorph} Retain the current setting. Assume there exists $x\in X$ with $\alpha_x=\beta_x$.

Then for any finite set $G\subseteq \Cu(\mathcal{A})$, there exists an open neighborhood $V$ of $x$ such that $\alpha_v\simeq_{G_v}\beta_v$ for all $v\in V$.
\end{lma}

\begin{proof}
We know that for any $g',g\in G$ such that $g'\ll g$, we have that $\alpha(g)(x)\leq\beta(g)(x)$ and $\alpha(g')\ll \alpha(g)$. By \cite[Lemma 2.2 (i)]{ABP13}, there exists an open neighborhood $V$ of $x$ such that $\alpha(g')(v)\leq \beta(g)(v)$ for all $v\in V$. A symmetric argument together with finiteness of $G$ end the proof.
\end{proof}

By a \emph{trivial field}, we mean a continuous field $(X,\{A_x\}_{x}, \mathcal{A})$ such that $\mathcal{A}$ is isomorphic to $C(X)\otimes A$, for some $\Ca$-algebra $A$. 

By a \emph{locally trivial field}, we mean a continuous field $(X,\{A_x\}_{x}, \mathcal{A})$ satisfying that for any $x\in X$, there exists an open neighborhood $V$ of $x$ such that the (closed two-sided) ideal $\mathcal{I}_U:= C_0(U)\mathcal{A}$ is isomorphic to $C_0(U)\otimes A_x$, for some $\Ca$-algebra $A_x$.

\begin{thm}\label{thm:lsc} Retain the current setting. Assume moreover that $(X,\{A_x\}_{x}, \mathcal{A})$ is a locally trivial field.
\begin{itemize}
\item[(i)] The family $\mathcal{S}_\alpha$ satisfies the uniform path property. See \autoref{thm:selection} (ii).

\item[(ii)] The map $\Theta_\alpha\colon X\longrightarrow \mathcal{S}_\alpha$ is lower-semicontinuous whenever $\Hom_{\mathcal{F}_{\Ca}}(\mathcal{A},\mathcal{B})$ satisfies $\rm (PP)$.
\end{itemize}
\end{thm}

\begin{proof}
(i) Recall that $A_x$ is interval algebras for any $x\in X$. Combined with the (local) triviality of $(X,\{A_x\}_{x}, \mathcal{A})$, we deduce the uniform path property from \cite[Theorem 2.3]{A99} that we will explicitly recall and use later. (See \autoref{thm:pepite}).

(ii) Let $V_{\epsilon}^\chi$ be a basic open set of $\tau_{(\mathcal{A}\rightarrow \mathcal{B})}$. We aim to show that $\{x\in X\mid \Theta_{\alpha}(x)\cap V_{\epsilon}^\chi\neq \emptyset\}$ is open in $X$. Let $x\in X$ be such that $\Theta_{\alpha}(x)\cap V_{\epsilon}^\chi\neq \emptyset$. It suffices to prove that there exists an open neighborhood $Y$ of $x$ such that $Y\subseteq \{x\in X\mid \Theta_{\alpha}(x)\cap V_{\epsilon}^\chi\neq \emptyset\}$. 

First, observe that $x\in V$ and that there exists $\phi_x\in \Theta_{\alpha}(x)$ such that $d_x(\phi_x,\chi_x)<\epsilon$. Choose $\epsilon'>0$ such that $d_x(\phi_x,\chi_x)<\epsilon'<\epsilon$. By $\rm (PP)$, we can find $\psi\in \Hom_{\mathcal{F}_{\Ca}}(\mathcal{A},\mathcal{B})$ such that $\psi_x=\phi_x$ and \autoref{lma:approxextendmorph} yields an open neighborhood $W$ of $x$ such that $d_w(\psi_w,\chi_w)<\epsilon'$ for all $w\in W$.

Now, fix $\delta>0$ small enough such that $\delta+\epsilon'<\epsilon$ and a finite set $F\subseteq \mathcal{A}$. Let $x\in X$. By the local triviality of $(X,\{A_x\}_{x}, \mathcal{A})$, we can  find an open neighborhood $Z$ of $x$ such that $A_z\simeq A_x$ for all $z\in Z$. Now, we apply the uniform continuity of the classification recalled in \autoref{prg:uniform}, to find a finite set $G\subseteq \Cu(\mathcal{A})$, such that $G_z$ is a $(F_z,\delta)$-admissible set for all $z\in Z$. Since $\Cu(\psi_x)=\alpha_x$, we can apply \autoref{lma:approxCumorph} to find an open neighborhood $\tilde{Z}$ of $x$ such that $\Cu(\psi_z)\simeq_{G_z}\alpha_z$ for all $z\in\tilde{Z}$. Also, recall that $\Theta_\alpha(z)$ is non-empty. Henceforth, for any $z\in Z\cap\tilde{Z}$, we are able to find $\phi_z\in \Theta_\alpha(z)$ such that $d_z(\psi_z,\phi_z)<\delta$. 

Putting everything together, we compute that
\[ d_z(\phi_z,\chi_z)\leq d_z(\phi_z,\psi_z)+ d_z(\psi_z,\chi_z)\leq \delta+\epsilon'<\epsilon
\]
for any $z\in W\cap Z\cap\tilde{Z}$. As a conclusion, for any $y\in Y:=W\cap Z\cap \tilde{Z}\cap V$, we have $\Theta_{\alpha}(y)\cap V_\epsilon^\chi\neq \emptyset$ which ends the proof.
\end{proof}

As a consequence, we can invoke Michael's selection theorem (see \autoref{thm:selection})  to obtain the following existence result.

\begin{cor}\label{cor:existence}
Let $A,B$ be unital interval algebras. For any scaled $\Cu$-morphism\break $\alpha\colon \Cu(C(X)\otimes A)\rightarrow \Cu(C(X)\otimes B)$ descending to the fibers, there exists $\phi\in \Hom_{\mathcal{F}_{\Ca}}(\mathcal{A},\mathcal{B})$ such that $\Cu(\phi)=\alpha$. 

A fortiori, $\Cu(\phi_x)=\alpha_x$ for any $x\in X$.
\end{cor}

\begin{proof} It is immediate that the triviality of $(X,\{A_x\}_{x}, \mathcal{A})$ and $(X,\{B_x\}_{x}, \mathcal{B})$ implies local triviality of $(X,\{A_x\}_{x}, \mathcal{A})$ and $\rm (PP)$ of $\Hom_{\mathcal{F}_{\Ca}}(\mathcal{A},\mathcal{B})$.  
The result now follows from \autoref{prop:(i)} and \autoref{thm:lsc}.
\end{proof}

We remark that the restriction to trivial continuous fields was imposed solely to obtain $\rm (PP)$. Consequently, establishing $\rm (PP)$ for a class of locally trivial fields would lead to a more general existence result.

\begin{prg}We thank M. D\u{a}d\u{a}rlat for pointing out that the Mayer-Vietoris sequence applied to $\K$-theory could provide an obstruction to $\rm (PP)$ for locally trivial fields over the circle. Let us construct a specific example.

Let $D:=M_{2^\infty}\oplus M_{2^\infty}$ be two copies of the CAR-algebra. Observe that $D$ is a unital $\AF$-algebra and that $\K_0(D)\simeq \Z[\frac{1}{2}]\oplus \Z[\frac{1}{2}]$. Also, by considering any *-isomorphism $\nu\colon M_2(M_{2^\infty})\simeq M_{2^\infty}$ (e.g. given by partial isometries adding to the unit), we construct an automorphism $\varphi\colon D\simeq D$ given by $(\nu\circ (a\mapsto \begin{psmallmatrix}a&0\\0&a\end{psmallmatrix}),\id_D)$. Write $\alpha:=\K_0(\varphi)$ and observe that the matrix representation of $\alpha$ is $\begin{psmallmatrix}2&0\\0&1\end{psmallmatrix}$. Now consider 
\[
\mathcal{A}:=C(\T)\otimes D \quad\quad \text{and}  \quad\quad \mathcal{B}:=\{f\in C([0,1],D)\mid f(1)=\varphi(f(0))\}.
\]
Observe that $(\T,\{D\},\mathcal{A})$ and $(\T,\{D\},\mathcal{B})$ are well-defined locally trivial continuous fields of $\Ca$-algebras with fibers $D$. (In fact, $\mathcal{A}$ is the trivial continuous field.) By the K\"{u}nneth formula applied to $\mathcal{A}$ and the Mayer-Vietoris sequence applied to the pullback $\mathcal{B}$, we obtain that $\K_0(\mathcal{A})\simeq\K_0(D)$ and $ \K_0(\mathcal{B})\simeq \Z[\frac{1}{2}]$. In particular, $\K_0(\mathcal{A})\nsimeq \K_0(\mathcal{B})$. 

Now, let $x\in \T$ and let $\phi_x\colon A_x\rightarrow B_x$ be the identity morphism on $D$. (That is, $\phi_x:=\id_D$.) 
We shall prove that $\phi_x$ cannot be approximated by any continuous field morphism $\chi\colon \mathcal{A}\rightarrow \mathcal{B}$ by contradiction. 

Assume that there exists $\chi\colon \mathcal{A}\rightarrow \mathcal{B}$ such that $\chi_x=\phi_x$. In particular, $\K_0(\chi_x)\colon \K_0(D)\rightarrow \K_0(D)$ is a group isomorphism. Now, observe that for any $y\in \T$, there exists a proper closed subinterval $V$ of $\T$ joining $x$ to $y$. Moreover, the restriction of $\mathcal{B}$ to $V$ is a trivial continuous field over $V$. Therefore $\chi_x\colon D\rightarrow D$ and $\chi_y\colon D\rightarrow D$ are homotopic *-homomorphisms. 

By homotopy invariance of $\K$-theory, we deduce that $\K_0(\chi_y)$ is a group isomorphism for any $y\in \T$. We conclude that for any proper closed subinterval $V$ of $\T$, the induced map $\K_0(\chi_{\mid V})\colon \K_0(\mathcal{A}(V))\rightarrow \K_0(\mathcal{B}(V))$ (where $\mathcal{A}(V)$ and $\mathcal{B}(V)$ denote the quotient $\Ca$-algebras obtained by restriction of the sections to $V$) is also a group isomorphism.

Finally, let $V,W$ be proper closed subintervals of $\T$ such that $V\cup W=\T$. It is well known that $\mathcal{A}\simeq\mathcal{A}(V)\oplus_{\mathcal{A}(V\cap W)} \mathcal{A}(W)$ and $\mathcal{B}\simeq\mathcal{B}(V)\oplus_{\mathcal{B}(V\cap W)} \mathcal{B}(W)$. (See e.g. \cite{AM10} for the sheaves picture of continuous fields.) 
The Mayer-Vietoris sequence yields $\K_0(\mathcal{A})\simeq \K_0(\mathcal{B})$ which leads to a contradiction. As a consequence, $\Hom_{\mathcal{F}_{\Ca}}(\mathcal{A},\mathcal{B})$ does not have $\rm (PP)$.
\end{prg}

\begin{qst}
Let $\mathcal{A},\mathcal{B}$ be continuous fields of interval algebras over the interval. Under which conditions $\Hom_{\mathcal{F}_{\Ca}}(\mathcal{A},\mathcal{B})$ has $\rm (PP)$? 
\end{qst}

\subsection{The uniqueness}
We aim to extend the uniqueness part of \autoref{thm:RR} to morphisms between continuous fields whose fibers fall into the class of separable unital $\AI$-algebras. As for the existence part, restrictions on the base space and the fibers will appear. The technical key is inspired from a \textquoteleft gluing process\textquoteright\ developed in \cite[Lemma 3.4]{DEN11}, used to extend from local-to-global, the classification of $\AF$-algebras by means of the $\K_0$-group to certain continuous fields whose fibers fall into the class of $\AF$-algebras. Nevertheless, their argumentation has to be refined for our context, which generally has $\K_1$-obstructions. Our gluing process is based on the following result.
\begin{thm}\cite[Theorem 2.3]{A99}
\label{thm:pepite}
Let $A$ be a unital interval algebra. Let $F\subseteq A$ be a finite set containing the canonical generators of $A$ and let $\epsilon>0$. Then there exists $\delta(F,\epsilon) > 0$ satisfying the following. 

For any $*$-homomorphism $\phi: A \to B$, where $B$ is any $\AI$-algebra, and any unitary $w\in B$ such that $w\phi w^*\simeq_{(F,\delta)}\phi$, there exists a path $\gamma$ in $U(B)$ joining $w$ to $1$ such that
\[
\gamma(s)\phi \gamma(s)^*\simeq_{(F,\epsilon)}\phi
\]
for any $0\leq s\leq 1$.
\end{thm}

\begin{proof}
This is essentially the same proof, combined with the fact that $F$ is finite and contains the generators of $A$.
\end{proof}

$\bullet$ \textbf{The interval case.} The gluing procedure is somehow technical and uses properties about the topological base space at hand. For the purpose of this paper, we focus on continuous fields over the closed interval $[0,1]$. For instance, it allows the following lemma which will appear useful to realize our gluing procedure.

\begin{lma}\label{lma:fiberwise}
Let $([0,1],\{A\}_{x}, \mathcal{A})$ be a continuous field of $\Ca$-algebras. Let $a,b\in \mathcal{A}$ be such that $a(y)=b(y)$ for some $y\in [0,1]$. 

Then the element $c\colon x\mapsto \biggl\{\begin{array}{ll} a(x) \text{ if }x \leq y\\ b(x) \text{ if }x \geq y\end{array}$ is well-defined and belongs to $\mathcal{A}$.
\end{lma}

\begin{proof}
This readily follows from the continuity criterion. See \autoref{prop:cc}. 
\end{proof}

\begin{thm}[The gluing procedure]
\label{thm:gluing}
Let $([0,1],\{A_x\}_{x}, \mathcal{A})$ and $([0,1],\{B_x\}_{x}, \mathcal{B})$ be continuous fields of separable unital $\AI$-algebras. Assume moreover that the fibers of  $\mathcal{A}$ are interval algebras. Let $Y=[y_0,y_1]$ be a fixed closed interval. For any finite set $F\subseteq \mathcal{A}$ and any $\epsilon>0$, there exists $\delta > 0$ satisfying the following. 

For any continuous field morphisms $\phi, \psi: \mathcal{A} \to \mathcal{B}$ and any  $w_0,w_1\in U(\mathcal{B})$ such that 
\[
(\ast) \quad\quad\quad(w_0\phi w_0^*)_y\simeq_{(F_y,\delta)} \psi_y \simeq_{(F_y,\delta)} (w_1\phi w_1^*)_y, \quad \forall y\in Y,
\]
then there exists $w \in U(\mathcal{B})$ with $w(y_0)=w_0(y_0)$ and $w(y_1)=w_1(y_1)$,  such that
\[
(w\phi w^*)_y\simeq_{(F_y,\epsilon)} \psi_y, \quad \forall y\in Y.
\]
\end{thm}

\begin{proof}
We first consider the unitary element $u := w_1^* w_0$ of $\mathcal{B}$. Combining the comparisons of $(\ast)$, we deduce that $(u\phi u^*)_y\simeq_{(F_y,2\delta)} \phi_y $ for all $y\in Y$, and a fortiori at the fiber $y_0$. Now choose $\delta>0$ small enough such that $2\delta< \min(\delta(F_{y_0},\epsilon/3),\epsilon/3)$ where $\delta(F_{y_0},\epsilon/3)$ is obtained from \autoref{thm:pepite}. 
As a consequence, there exists a path $\gamma_0$ in $U(B_{y_0})$ joining $u(y_0)$ to $1_{B_{y_0}}$ such that
\[
\gamma_0(s)\phi_{y_0} \gamma_0(s)^*\simeq_{(F_{y_0},\epsilon/3)}\phi_{y_0}
\]
for any $0\leq s\leq 1$.

Since $\mathcal{B} \twoheadrightarrow B_{y_0}$ is a surjective unital *-homomorphism, it is known that elements of $U_0(B_{y_0})$ lift to elements of $U_0(\mathcal{B})$. (See e.g. \cite[Corollary 4.3.3]{WO93}.) Therefore, there exists a path $\gamma$ in $U_0(\mathcal{B})$ joining $u$ to $1_\mathcal{B}$ whose projection to the fiber $B_{y_0}$ satisfies that $\gamma_{y_0}=\gamma_0$.
By the continuity of the norm and the finiteness of $F$, we can find some $y_{\epsilon}\in Y$ such that 
\[
\gamma_y(s)\phi_{y} \gamma_y(s)^*\simeq_{(F_{y},2\epsilon/3)}\phi_{y} 
\]
for any $0\leq s\leq 1$ and any  $y\in [y_0,y_\epsilon]$.

Next we use \autoref{lma:fiberwise} to fiber-wisely define the following unitary element of $\mathcal{B}$
\[
\begin{array}{ll}
\widetilde{u}\colon [0,1] \longrightarrow \bigsqcup_xB_x\\
\hspace{1cm}x\longmapsto\Biggl\{\begin{array}{ll}
u(x) & \text{ if}\quad 0 \leq x \leq y_0 \\
\gamma_x(h(x)) &\text{ if}\quad y_0 < x \leq y_\epsilon \\
1_{B_y} &\text{ if}\quad y_\epsilon < x \leq 1
\end{array}
\end{array}
\]
where $h$ is the linear parametrisation satisfying $h(y_0)=0$ and $h(y_\epsilon)=1$.

Finally, we define the unitary element $w:=w_1\widetilde{u}$ of $\mathcal{B}$ that will complete our proof. We check that 
$w(y_0)=w_1(y_0)\widetilde{u}(y_0)=w_1(y_0)u(y_0)=w_0(y_0)$
and that $w(y_1)=w_1(y_1)\widetilde{u}(y_1)=w_1(y_1)$.
Moreover, for all $ y \in [y_0, y_1] $ and for all $ f \in F$, we compute that
\begin{align*}
\| (w \phi(f) w^*)(y) - \psi(f)(y)) \| &\leq \|  (w_1(\widetilde{u}\phi(f)\widetilde{u}^*-\phi(f))w_1^*)(y) \|
+ \|  (w_1 \phi(f) w_1^* - \psi(f))(y) \| \\
&\leq 2 \epsilon/3 + \delta\leq \epsilon
\end{align*}
which ends the proof.
\end{proof}

\begin{thm}\label{thm:uniqueness}
Let $([0,1],\{A_x\}_{x}, \mathcal{A})$ and $([0,1],\{B_x\}_{x}, \mathcal{B})$ be continuous fields of separable unital $\AI$-algebras.  Assume moreover that the fibers of  $\mathcal{A}$ are interval algebras. 

Let $\phi, \psi: \mathcal{A} \to \mathcal{B}$ be continuous field morphisms. For any finite set $F\subseteq \mathcal{A}$ and any $\epsilon>0$, there exists a finite set $G\subseteq \Cu(\mathcal{A})$ satisfying the following. 

If $\Cu(\phi_x)\simeq_{G_x}\Cu(\psi_x)$ for any $x\in [0,1]$, then there exists a unitary element $w\in U(\mathcal{B})$ such that
$w\phi w^*\simeq_{(F,\epsilon)} \psi$.
\end{thm}

\begin{proof}
Let $x\in X$ and let $\delta>0$. By the uniform continuity of the classification of \autoref{thm:RR}, we can find a subset $H\subseteq\Cu(\mathcal{A})$ such that $H_x$ is $(F_x, \delta)$-admissible for $A_x$. (See \autoref{prg:uniform}.) Using lifting of unitaries arguments similarly as in the proof of \autoref{thm:gluing}, there exists a unitary element $w\in U_0(\mathcal{B})$ such that $(w\phi w^*)_x\simeq_{(F_x,\delta)}\psi_x$, whenever $\Cu(\phi_x)\simeq_{H_x}\Cu(\psi_x)$. Further, by \autoref{lma:approxextendmorph}, we can find an open  an open neighborhood $V$ of $x$ such that $ (w\phi w^*)_v\simeq_{(F_v,\delta)} \psi_v$ for any $ v\in V$, whenever $\Cu(\phi_x)\simeq_{H_x}\Cu(\psi_x)$. 

Subsequently, it follows from the compactness and the dimension of $[0,1]$ that there exist finitely many points $x_1,\dots,x_l$ together with open neighborhoods $V_1,\dots,V_l$, unitary elements $w_1,\dots,w_l$ of $\mathcal{B}$ and finite sets $G^{(1)},\dots,G^{(l)}$ of $\Cu(\mathcal{A})$, satisfying the followings.

$\bullet$ $\bigcup_1^l V_i$ is a cover of $[0,1]$ such that $V_i\cap V_j=\emptyset $ whenever $|i-j|>1$.

$\bullet$ For any $1\leq i\leq l$ and any $v\in V_i$, we have that
$(w_i\phi w_i^*)_v\simeq_{(F_v,\delta)} \psi_v $, whenever $\Cu(\phi_{x_i})\simeq_{G^{(i)}_{x_i}}\Cu(\psi_{x_i})$.

Set $G:=\bigcup_1^l G^{(i)}$ and assume that $\Cu(\phi_x)\simeq_{G_x}\Cu(\psi_x) $ for all $x$. In particular, for any $1\leq i\leq l$,	 we have $\Cu(\phi_{x_i})\simeq_{G^{(i)}_{x_i}}\Cu(\psi_{x_i}) $ and hence, $(w_i\phi w_i^*)_v\simeq_{(F_v,\delta)} \psi_v $ any $v\in V_i$.

Let us recursively modify the unitary elements $\{\omega_i\}_1^l$ in order to construct the global unitary we are looking for. For any $1\leq i\leq l-1$, we fix a subinterval $Y_i:=[y_{i,0},y_{i,1}]\subsetneq V_i\cap V_{i+1}$ and $\delta>0$ small enough (in particular, we choose $\delta<\epsilon$) in order to be able to invoke \autoref{thm:gluing} and obtain a unitary element $\widetilde{w}_i$ of $\mathcal{B}$ satisfying the following
\[
\biggl\{\begin{array}{ll}
(\widetilde{w}_i\phi\widetilde{w}_i^*)_y\simeq_{(F_y,\epsilon)}\psi_y,\quad \forall y\in Y_i\\
\widetilde{w}_i(y_{i,0})=w_i(y_{i,0}) \text{ and } \widetilde{w}_i(y_{i,1})=w_{i+1}(y_{i,1})
\end{array}
\]
We finally use Lemma \ref{lma:fiberwise} to fiber-wisely define the following unitary element of $\mathcal{B}$
\[
\begin{array}{ll}
w\colon [0,1] \longrightarrow \bigsqcup_x B_x\\
\hspace{1cm}x\longmapsto\biggl\{\begin{array}{ll}
w_i(x) & \text{ if}\quad x\in V_i\setminus Y_i \\
\widetilde{w}_i(x) &\text{ if}\quad x\in Y_i
\end{array}
\end{array}
\]
 from which we conclude that 
$(w\phi w^*)_x\simeq_{(F_x,\epsilon)} \psi_x$ for any $x\in \bigcup_1^l V_i=[0,1].$
\end{proof}

\begin{cor}\label{cor:uniqueness}
Let $([0,1],\{A_x\}_{x}, \mathcal{A})$ and $([0,1],\{B_x\}_{x}, \mathcal{B})$ be continuous fields of separable unital $\AI$-algebras.  Assume moreover that the fibers of  $\mathcal{A}$ are interval algebras. Let $\phi, \psi: \mathcal{A} \to \mathcal{B}$ be continuous field morphisms such that $\Cu(\phi_x)=\Cu(\psi_x)$ for all $x\in [0,1]$. 

Then, $\phi$ and $\psi$ are approximately unitarily equivalent as *-homomorphisms. That is, there exists a sequence $(w_n)_n$ in $U(\mathcal{B})$ such that
\[
w_n\phi w_n^*(a)\underset{n\rightarrow\infty}{\longrightarrow} \psi(a), \quad \forall a\in \mathcal{A}
\]
\end{cor}

\begin{proof}
This is an immediate consequence of \autoref{thm:uniqueness}.
\end{proof}

\begin{qst}
Could we mimic the above gluing procedure for continuous fields over the circle?
\end{qst}

A positive answer would most likely yield a uniqueness result for continuous fields of separable unital $\AI$-algebras over any one-dimensional compact CW-complexes, which may be generalized to any one-dimensional compact Hausdorff metric space.

%Before diving into the gluing process, we need the following intermediate lemma. Assume that $X$ is a one-dimensional compact (Hausdorff) CW-complex. We can construct a finite open cover $\{V_k\}_1^n$ such that $V_k\cap V_l$ is either empty or homeomorphic to an interval and such that any $x\in X$ is contained in at most two sets. In that case, there are exactly $l-1$ non-empty intersections that we write $\{Y_k\}_1^{l-1}$. Now, let us reorder the cover as follows. 
%
%Notice that for each $1\leq k\leq n-1$, we have $(0,1)\overset{\nu_k}{\simeq} Y_k$. Denote $V_{j_k}$ to be the unique open set of the cover such that $V_{j_k}\cap{\nu_k(0)}$ is non-empty and $V_{j_n}$ to be the unique remaining one.
%
%\begin{lma}
%\label{lma:fiberwise}
%Let $(X,\{A_x\}_{x}, \mathcal{A})$ be a continuous field. Assume that $X$ is a one-dimensional compact (Hausdorff) CW-complex. Let $\{V_{j_k}\}_1^n$ and $\{Y_k\}_1^{n-1}$ be as above. Fix a finite sequence  $\{y_k\}_1^{n-1}$ in $X$ such that $y_k\in Y_k$ for all $1\leq k\leq n-1$.
%Let $\{a_k\}_1^n$ be a finite sequence in $\mathcal{A}$ be such that $a_k(y_k)=a_{k+1}(y_k)$.
%
% For any $1\leq k\leq n-1$ and any $x\in (V_{j_k}\setminus \overline{\bigcap_{j\neq k} V_j})\bigcup \nu_k([0,y_k])$, define $c_x:=a_k(x)$. For $k=n$ and any $x\in (V_{j_n}\setminus \overline{\bigcap_{j\neq k} V_j})\bigcup \nu_k([y_k,1])$, define $c_x:=a_n(x)$. 
% 
% Then the element $c: X \to \bigsqcup_x A_x$ defined by $x\mapsto c_x$ belongs to $\mathcal{A}$.
%\end{lma}
%
%\begin{proof}
%This readily follows from the continuity criterion. See \autoref{prop:cc}. 
%\end{proof}

\subsection{The classification} Gathering all of the above, we are able to establish classification results regarding continuous fields over the interval and their morphisms. The machinery will be analogous to the one used for *-homomorphisms and $\Ca$-algebras. Therefore, we explicitly introduce notions about continuous fields morphisms alike to the one for *-homomorphisms. 

Firstly, we say that two continuous field morphisms $\phi,\psi\colon (X,\{A_x\}_{x}, \mathcal{A})\rightarrow (X,\{B_x\}_{x}, \mathcal{B})$ are \emph{approximately unitarily equivalent}, if there exists a sequence $(w_n)_n$ of unitaries of $\mathcal{B}$ such that $w_n\phi w_n^*\underset{n\rightarrow\infty}{\longrightarrow} \psi$ in the point-norm topology. Equivalently, if they are approximately unitarily equivalent as *-homomorphisms.

Secondly, we say that a pair $[(X,\{A_x\}_{x}, \mathcal{A}), (X,\{B_x\}_{x}, \mathcal{B})]$ of continuous fields satisfies the \emph{classification property}, termed (CP) for short, if for any scaled $\Cu$-morphism $\alpha\colon \Cu(\mathcal{A})\rightarrow \Cu(\mathcal{B})$ descending to the fibers, there exists a continuous field morphism $\phi\colon \mathcal{A}\rightarrow \mathcal{B}$, unique up to approximate unitary equivalence, such that $\Cu(\phi)=\alpha$. We may write \textquoteleft$[\mathcal{A},\mathcal{B}]$ has (CP)\textquoteright\ when the context is clear.

Lastly, the following property will allow us to use a similar machinery as in the category of $\Ca$-algebras.

\begin{prop}
Let  $(X,\{(A_i)_x\}_x,\mathcal{A}_i)_i$ be a family of continuous fields indexed over an upward-directed set. Consider any inductive system $((X,\{(A_i)_x\}_x,\mathcal{A}_i),\phi_{i,j})_i$ in $\mathcal{F}_{\Ca}$. We have the following.

(i) $\mathcal{F}_{\Ca}\text{-}\lim (\mathcal{A}_i,\phi_{i,j})\simeq\Ca\text{-}\lim (\mathcal{A}_i,\phi_{i,j})$.

(ii) $[\mathcal{F}_{\Ca}\text{-}\lim (\mathcal{A}_i,\phi_{i,j}),\mathcal{B}]$  has (CP) whenever $[\mathcal{A}_i, \mathcal{B}]$ has (CP) for all $i\in I$.
\end{prop}

\begin{proof}
(i) Consider the inductive system in the category $\Ca$ and denote by $(\mathcal{A},\phi_{i,\infty})$ its $\Ca$-limit. For any $x\in X$, consider the inductive system obtained by the projection of the objects and morphisms at the fiber $x$. We obtain another inductive system in $\Ca$ and we denote its limit by $(A_x,(\phi_{i,\infty})_x)$. A standard one-sided intertwining argument shows that there exists a unique *-homomorphism $\pi_x\colon \mathcal{A}\rightarrow A_x$ compatible with the squares of the intertwining. A fortiori, $(X,\{A_x\}_x,\mathcal{A})$ is a well-defined continuous field over $X$ and $\phi_{i,\infty}$ is a continuous field morphism for any $i\in I$. It is now readily checked that $(\mathcal{A},\phi_{i,\infty})$ satisfies the universal property of inductive limits in the category $\mathcal{F}_{\Ca}$.

(ii) This is done similarly as in the proof of \cite[Proposition 5.2 (iv)]{CES11}.
\end{proof}

In what follows, $\mathcal{C}$ denotes the class of continuous fields that are realized as $\underset{\rightarrow}{\lim}(C[0,1]\otimes A_i,\phi_{i,i+1})_{i\in\N}$, where $A_i$ are unital interval algebras and $\phi_{i,i+1}$ are unital continuous fields morphisms.

We emphasize that $\mathcal{C}$-objects do not reduce to trivial fields, i.e. of the form $C[0,1]\otimes A$, where $A$ is an $\AI$-algebra. 
Similarly, $\mathcal{C}$-morphisms do not reduce to trivial morphisms, i.e. of the form $\id_{[0,1]}\otimes\, \phi\colon C[0,1]\otimes A\rightarrow C[0,1]\otimes B$, where $\phi\colon A\rightarrow B$ is a *-homomorphism between two interval algebras $A,B$. (See e.g. \cite{DP89}.)

\begin{thm}\label{thm:classification}
Let $\mathcal{A},\mathcal{B}\in \mathcal{C}$. Let $\alpha\colon \Cu(\mathcal{A})\rightarrow \Cu(\mathcal{B})$ be any scaled $\Cu$-morphism descending to the fibers.

Then there exists a continuous field morphism $\phi\colon \mathcal{A}\rightarrow\mathcal{B}$, unique up to approximate unitary equivalence, such that $\Cu(\phi)=\alpha$. 
A fortiori, $\Cu(\phi_x)=\alpha_x$ for any $x\in [0,1]$.
\end{thm}

\begin{proof}
Combine \autoref{cor:existence} with \autoref{cor:uniqueness} and the fact that $[\mathcal{A},\mathcal{B}]$ has (CP).
\end{proof}

\begin{cor}
Let $\mathcal{A},\mathcal{B}\in \mathcal{C}$. Then any $\Cu$-isomorphism $\alpha\colon\Cu(A)\simeq \Cu(B)$ descending to the fibers and such that $\alpha([1_A])=[1_B]$, lifts to a unique (up to approximate unitary equivalence) continuous field isomorphism $\mathcal{A}\simeq \mathcal{B}$.
\end{cor}

\begin{proof}
In the light of the above remarks, a standard approximate intertwining argument yields the result from the above theorem.
\end{proof}

\begin{rmk}[Computational use of the classification] The Cuntz semigroup of a $\Ca$-algebra (and a fortiori of a continuous field) is often very hard to compute. Some results in this direction can be found e.g., in \cite{BPT08, R07, APS11}. In our opinion, it seems unlikely that the computation of the Cuntz semigroup of a $\mathcal{C}$-object can be recovered from the one of its fibers, since $\K_1$-obstructions will appear in general. (See \cite[Theorem 3.4]{APS11}.)

Nevertheless, we believe that the classification theorem obtained could be used to distinguish or relate explicit of continuous fields via approximate intertwining methods similarly as employed in e.g. \cite{C23,C24}.
\end{rmk}

\section{An outro on continuous fields of Cuntz semigroups}
The final theorem and corollary of the previous section classify a specific class continuous fields of $\Ca$-algebras and its morphisms. It has been common to express such results in a categorical fashion, via a well-defined functor $F\colon \mathcal{C}\rightarrow \mathcal{D}$. (See e.g. \cite[Section 2]{C25}.) This section suggests a categorical context $\mathcal{F}_{\Cu}$ of \emph{continuous fields of Cuntz semigroups}, suitable to construct a well-defined classifying functor. 

Let $X$ be a compact Hausdorff space. In the aim of defining an analogous notion of continuous fields for $\Cu$-semigroups, we ought to recall some facts about (arbitrary) products in the category $\Cu$. We refer the reader to \cite{APT20a,APT20b} for more on this matter. Let $\{S_x\}_x$ be a family of $\Cu$-semigroups indexed over $X$. We can consider the following product in the category of Set.
\[
Q:=\prod_{x\in X} S_x
\]
We can equip $Q$ with point-wise sum and order, so that $Q$ becomes a positively ordered monoid closed under suprema of increasing sequences. Unfortunately, the compact-containment relation $\ll$ obtained from the point-wise order does not turn $Q$ into a $\Cu$-semigroup, since (O2) fails in general. For that matter, the \emph{$\tau$-construction} has been introduced in \cite{APT20a}. More specifically, they consider a slightly larger category $\mathcal{Q}$ consisting of positively ordered monoid closed under suprema of increasing sequence, equipped with an \emph{auxiliary relation} $\prec$, and their morphisms. A fortiori, any $\Cu$-semigroup $S$ is a $\mathcal{Q}$-semigroup when equipped with the compact-containment relation. In other words, $(S,\ll)\in \mathcal{Q}$ for any $S\in \Cu$. 

Subsequently, they construct functor $\tau\colon \mathcal{Q}\rightarrow \Cu$ which sends a $\mathcal{Q}$-semigroup $(Q,\prec)$ to a $\Cu$-semigroup $\tau(Q,\prec)$, which roughly consists of equivalence classes of \emph{$\prec$-increasing \textquoteleft paths\textquoteright}. See \cite{APT20a, APT20b} for more details on the $\tau$-construction. The functor $\tau$ is a left-adjoint of the inclusion functor $\nu\colon \Cu\rightarrow \mathcal{Q}$. In particular, $\tau$ passes arbitrary products from $\mathcal{Q}$ to $\Cu$. 

Back to our setting, the point-wise compact-containment relation $\ll_{\rm{pw}}$ is a well-defined auxiliary relation on $Q$. As a consequence, we obtain the (arbitrary) product of the family $\{S_x\}_x $ in the category $\Cu$ by applying the functor $\tau$ to $Q$. More specifically, we have
\[
\Cu-\!\prod_{x\in X} S_x:=\tau(Q,\ll_{\rm{pw}})
\]
\begin{rmk}
In specific cases, the arbitrary product $\prod_{x\in X} S_x$ is already a $\Cu$-semigroup. In this case, $\Cu-\prod_{x\in X} S_x\simeq \prod_{x\in X} S_x$ and we do not have to consider equivalence classes of $\ll_{\rm{pw}}$-increasing paths.
\end{rmk}

Finally, based on the ideas developed in \cite{ABP13} for \emph{(pre)sheaves of $\Cu$-semigroups}, we define a notion of continuous fields of $\Cu$-semigroups as follows. Again by a \emph{section}, we mean an element of $\Cu-\prod_x S_x$. 

\begin{dfn}\label{dfn:ctsfieldsCu}
A \emph{continuous field of $\Cu$-semigroups} is a triple $(X,\{S_x\}_{x}, \mathcal{S})$ where $X$ is a compact Hausdorff space, $\{S_x\}_{x}$ is a family of $\Cu$-semigroups indexed over $X$ and $\mathcal{S}$ is a family of sections satisfying the following conditions.

(i) $\mathcal{S}$ is a $\Cu$-semigroup under point-wise sum and order.

(ii) For each $x\in X$, the map $\pi_x\colon \mathcal{S}\rightarrow S_x$ sending $s\mapsto s(x)$ is a surjective $\Cu$-morphism.

(iii) For each $s',s,t\in \mathcal{S}$ such that $s(x)\leq t(x)$ for some $x\in X$ and $s'\ll s$ in $\mathcal{S}$, there exists an open neighborhood $V$ of $x$ such that $s'(y)\leq t(y)$ for all $y\in V$.

(iv) $\mathcal{S}$ is stable by $\Lsc(X,\overline{\N})$-multiplication. 
That is, for any $\lambda\in \Lsc(X,\overline{\N})$ and any $s=\overline{p_s}\in \mathcal{S}$, where $p_s$ is any path representative of $s$, the section $\lambda.s:=\overline{\lambda.p_s}$ belongs to $\mathcal{S}$, where $\lambda.p_s$ is the path in $\prod_{x\in X} S_x$ defined by $(\lambda(x).p_s(x))_x$.
 \end{dfn}

\begin{rmk}
Even though we do not pursue the study of abstract continuous fields of $\Cu$-semigroups here, let us point out that \autoref{dfn:ctsfieldsCu} (iii) has been modeled after the ideas of \cite[Lemma 2.2]{ABP13} and thoroughly thought has an analogous of \autoref{dfn:ctsfields} (iii). 

E.g., it seems plausible that an analogous topology $\tau_{\mathcal{S}}$ on the disjoint union $\bigsqcup_x S_x$ of the fibers, constructed analogously as in \cite[Section 2]{ABP13} would yield a continuity criterion.
\end{rmk}
 
We naturally define a \emph{continuous field morphism} between continuous fields of $\Cu$-semi\-groups $(X,\{S_x\}_x,\mathcal{S})$ and $(X,\{T_x\}_x,\mathcal{T})$ to be a $\Cu$-morphism $\alpha\colon \mathcal{S}\rightarrow \mathcal{T}$ which \emph{descends to the fibers}. In other words, $\alpha$ induces a family $\{\alpha_x\colon S_x\rightarrow T_x\}$ of $\Cu$-morphisms such that $\pi_x^\mathcal{T}\circ\alpha=\alpha_x\circ\pi_x^\mathcal{S}$ for any $x\in X$. 

As a consequence, we define the category $\mathcal{F}_{\Cu}$ of continuous fields of $\Cu$-semigroups and their morphisms.

\begin{prop} The assignment 
\[
\begin{array}{ll}
	\hspace{1cm} \Cu\colon \mathcal{F}_{\Ca}\longrightarrow \mathcal{F}_{\Cu}\\
	\hspace{0,1cm} (X,\{A_x\}_{x}, \mathcal{A})\longmapsto (X,\{\Cu(A_x)\}_{x}, \Cu(\mathcal{A}))\\
	\hspace{2,2cm} \phi \longmapsto \Cu(\phi)
\end{array}
\] 
is a well-defined functor.
\end{prop}

\begin{proof} Let $(X,\{A_x\}_{x}, \mathcal{A})$ be a continuous field of $\Ca$-algebras. Observe that $\Cu(A_x)$ and $\Cu(\mathcal{A})$ are $\Cu$-semigroups and that for any $x\in X$, we have a surjective $\Cu$-morphism $\Cu(\pi_x)\colon \Cu(\mathcal{A})\rightarrow \Cu(A_x)$. Also, There is a natural order-embedding $\nu\colon \Cu(\mathcal{A})\lhook\joinrel\rightarrow \Cu-\prod_x \Cu(A_x)$ given by $[a]\mapsto \overline{p_a}$ where $p_a$ is the $\ll_{\rm{pw}}$-increasing path of $\prod_x \Cu(A_x)$ given by $\epsilon\mapsto (\Cu(\pi_x)([(a-\epsilon)_+]))_x$. 
Next, we deduce (iii) of \autoref{dfn:ctsfieldsCu} from \cite[Lemma 2.2 (i)]{ABP13}, while we left (iv) for the reader to check. 

The fact that a $\mathcal{F}_{\Ca}$-morphism $\phi$ induces a $\mathcal{F}_{\Cu}$-morphism $\Cu(\phi)$ falls from construction.
\end{proof}

We shall leave a more in-depth study of the category $\mathcal{F}_{\Cu}$ and its associated functor $\Cu$ (e.g. its continuity) for future works.
We end the manuscript by rewriting the classification results obtained in the previous section in categorical terms.

\begin{thm} The functor $\Cu\colon \mathcal{F}_{\Ca}\longrightarrow \mathcal{F}_{\Cu}$ classifies continuous fields morphisms of $\mathcal{C}$. 

A fortiori, any $\mathcal{F}_{\Cu}$-isomorphism $\alpha\colon \Cu((X,\{A_x\}_{x}, \mathcal{A}))\rightarrow \Cu((X,\{B_x\}_{x}, \mathcal{B}))$ lifts to a unique (up to approximate unitary equivalence) $\mathcal{F}_{\Ca}$-isomorphism $\phi\colon (X,\{A_x\}_{x}, \mathcal{A})\rightarrow (X,\{B_x\}_{x}, \mathcal{B})$. 
\end{thm}

\end{document}